\documentclass{svproc}
\usepackage{url}
\usepackage{amssymb,amsmath,euscript}
\usepackage{cite}
\usepackage{paralist}
\usepackage{enumitem}
\usepackage[]{graphicx}\usepackage[]{color}
\usepackage[colorlinks=true]{hyperref}
\hypersetup{citecolor=blue, urlcolor=blue, linkcolor=blue, filecolor=red}
\numberwithin{equation}{section}
\usepackage{bm}

\DeclareMathOperator{\Ker}{Ker}

\def\cA{\EuScript{A}}

\def\cM{\EuScript{M}}

\def\cR{\EuScript{R}}

\usepackage{authblk}

\begin{document}
\mainmatter              % start of a contribution
\title{\bf Linearization for difference equations with infinite delay}
\titlerunning{ Linearization for delay difference equations }  % abbreviated title (for running head)
%                                     also used for the TOC unless
%                                     \toctitle is used
%
\author{Lokesh Singh\inst{1}
}

\authorrunning{Lokesh Singh} % abbreviated author list (for running head)
%
%%%% list of authors for the TOC (use if author list has to be modified)
%\tocauthor{Ivar Ekeland, Roger Temam, Jeffrey Dean, David Grove,
%Craig Chambers, Kim B. Bruce, and Elisa Bertino}
%
\institute{University of Rijeka, Croatia.\\
\email{lokesh.singh@uniri.hr},}

\maketitle              % typeset the title of the contribution

\begin{abstract}
In this article, we construct a conjugacy map for a linear difference equation with infinite delay and corresponding nonlinear perturbation. We also prove that the conjugacy map is one one with some additional conditions. As an application of our result, we show that the cases of (uniform) exponential dichotomy follow from our result.

\keywords{Delay Difference equations, Hartman-Grobman theorem,  Linearization.}
\end{abstract}
\section{Introduction}\label{sec:introduction}

The classical Hartman and Grobman theorem \cite{HP,GD} is a fundamental result in the local theory of differential equations and dynamical systems. This celebrated theorem provides a topological conjugacy (near a hyperbolic equilibrium point $x^*=0$) between the dynamics of a nonlinear differential equation $x'= Ax + f(x)$ (on a finite-dimensonal space) and the dynamics  corresponding to the  linear equation $x'= Ax$.  Later, this result was generalized by Pugh \cite{P} and Palis~\cite{Palis} to the  Banach space setting.  Some other improvements of the theorem are due to Reinfelds \cite{AR}, and Bates and Lu \cite{BL}. It is well known that in general, the conjugacy in the Grobman-Hartman theorem is  only locally Hölder continuous. On the other hand, there has been a significant amount of work concerned with formlating sufficient conditions under which the conjugacy exhibits higher regularity properties. Some of the early work in this direction is by Sternberg \cite{SS} and Belitskii \cite{BG,BG2}, while Rodrigues et. al. \cite{RMM}, ElBialy \cite{EM} and Zhang and Zhang along with Lu \cite{ZZ1,ZZ2,ZLZ} contributed recently.

 Palmer \cite{PK} proved the first version of the Grobman-Hartman theorem for nonautonomous differential equations by  assuming that the linear part admits an exponential dichotomy (see~Subsection~(\ref{exp_dich}) for definition). Later Shi and Xiong \cite{SXK} obtained an improvement on Palmer's result and established Hölder continuity of conjugacies. The version of Palmer's theorem in the case of discrete time was first  established by Aulbach and Wanner \cite{AW}. Castañeda and Robledo \cite{CR} discussed the regularity of conjugacy map for nonautonomous differential equations. More recently, Barreira and Valls \cite{BC}  dealth with the case when the linear part admits a nonuniform exponential dichotomy. 
 More recently, Dragičević, Zhang and Zhang \cite{DZZ,DZZ2} discussed the higher regularity of conjugacies in the nonautonomous setting. 
 Some other important contributions to nonautonomous linearization are given in~\cite{DB, DBV, DD}.

For delay differential equations, due to several complexities, there has not been much progress. As Sternberg mentioned in \cite{SN2}, in the case of delay differential equations, the solutions form semiflows (instead of flows) and sometimes solutions may not exist in backward time. Later, he obtained Hartman-Grobman theorem \cite{SN} for finite delay differential equations in the  finite-dimensional setting under some restrictive conditions. Namely, he assumed  the existence of a compact global attractor on which the semiflow is one-to -one. Benkhalti and Ezzinbi \cite{BE} obtained some improvement over Sternberg's result. Farkas \cite{FG}, proved the Hartman-Grobman theorem for autonomous delay differential equations admitting uniform exponential dichotomy in finite-dimensional settings. In recent work, Barreira and Valls \cite{BV} extended the result in continuous case of nonautonomous differential equations with finite delay in Banach space setting. They also assumed that the linear part admits a  uniform exponential dichotomy.  

In this work, our objective is to obtain a conjugacy map for nonautonomous discrete equations with infinite delay and corresponding nonlinear delay equations given in~(\ref{dis lin eq}) and~(\ref{dis semilin eq}) respectively. To obtain the conjugacy map, we assume a sufficient condition in terms of a Green type function (given in~(\ref{def: G})) associated with linear delay equation. With some additional assumptions, we also proved that the obtained conjugacy map is one-one. As an applications of our result, we showed that if the linear delay equation admits (uniform) exponential dichotomy, then the corresponding Green type function satisfies the assumptions of our Theorem~\ref{thm_2} and the result follows. With respect to existing work, our work has novelty in two senses, we are considering infinite delay difference equation and we proved the result with more general conditions.

This article is organized as follows. We describe our setup in Section \ref{pri}. In Section \ref{main}, we prove our main result and give an application of Theorem~\ref{thm_2}.

\section{Preliminaries}\label{pri}

Let $\mathbb{Z},\mathbb{Z}^+$ and $\mathbb{Z}^-$ denote the set of all integers, set of nonnegative integers and set of nonpositive integers, respectively. Let $(X,|\cdot|)$ be an arbitrary Banach space. Given a sequence $x: \mathbb{Z} \to X$ and $m \in \mathbb{Z}$, we define $x_m: \mathbb{Z}^- \to X$ by 
\begin{align*}
    x_m(j) = x(m+j) \quad \mbox{for all } j \in \mathbb{Z}^-.
\end{align*}
Next we consider a Banach space $(\mathcal{B},\| \cdot\|_{\mathcal{B}})$ of all sequences $\phi: \mathbb{Z}^- \to X$ satisfying the following assumption:
\begin{enumerate}
    \item[(A)] There exists $J>0$, and $K,M :\mathbb{Z}^+ \to  [0,\infty)$, such that if $x : \mathbb{Z} \to X$ with $x_0 \in \mathcal{B}$, then for all $n \in \mathbb{Z}^+$,
    \begin{align*}
        x_n \in \mathcal{B} \quad \mbox{and} \quad J |x(n)| \leq \|x_n\|_{\mathcal{B}} \leq K(n) \sup_{0 \leq j \leq n} |x(j)| + M(n) \|x_0\|_{\mathcal{B}}.
    \end{align*}
\end{enumerate}
As noted in \cite{SM}, one such space satisfying the assumption (A) is Banach space $(\mathcal{B}^\beta, \|\cdot \|_{\mathcal{B}^\beta})$, defined by
\begin{align*}
    \mathcal{B}^\beta = \Big\{ \phi: \mathbb{Z}^- \to X \Big| \|\phi\|_{\mathcal{B}^\beta} < \infty   \Big\}, \quad \|\phi\|_{\mathcal{B}^\beta} := \sup_{j \in \mathbb{Z}^-} [\,|\phi(j)| e^{\beta j} ] ,
\end{align*}
where $\beta$ is a real constant.
 Observe that the Banach Space $(\mathcal{B}^\beta, \|\cdot \|_{\mathcal{B}^\beta})$ satisfies (A) with $J= 1, M(n)= e^{-\beta n}$ and $K(n)=1$ if $\beta \geq 0$ and $K(n)= e^{-\beta n}$ if $\beta < 0.$

Given a sequence of linear operators $A_m : \mathcal{B} \to X, \, m \in \mathbb{Z}^+$, we consider a linear difference equation given by,
\begin{align}\label{dis lin eq}
     x(m+1) = A_mx_m, \qquad \mbox{for all } m \in \mathbb{Z}^+. 
\end{align}
Given $(n,\phi) \in \mathbb{Z}^+ \times \mathcal{B}$, there exists a unique sequence $x: \mathbb{Z} \to X$ such that $x_n= \phi$ and the sequence $x(m)$  satisfies equation (\ref{dis lin eq}) for all $m \geq n \geq 0$. The sequence $x$ is called a solution of equation (\ref{dis lin eq}) through $(n,\phi)$ and is denoted by $x = x(\cdot ; n,\phi)$. Now we define a two parameter solution operator $\cA(m,n):\mathcal{B} \to \mathcal{B}$ of the linear equation (\ref{dis lin eq}) by,
\begin{align}\label{sol op}
   \cA(m,n) \phi = x_m(\cdot ; n, \phi) \qquad \mbox{ for all } m \geq n \geq 0, \mbox{ and } \phi \in \mathcal{B}. 
\end{align}
Clearly, $(\cA(m,n))_{m \geq n \geq 0}$ is a discrete evolution family corresponding to linear equation (\ref{dis lin eq}) and it satisfies
 \begin{align*}
     \cA(n,n) &= Id_\mathcal{B} \qquad \quad \mbox{for all } n \geq 0, \\
     \cA(m,k) \cA(k,n) &= \cA(m,n) \quad \mbox{ for all } m \geq k \geq n \geq 0.
 \end{align*}Here $Id_\mathcal{B}$ denotes the identity operator on $\mathcal{B}$. Next we consider a sequence of projection maps $(P_n)_{n \in \mathbb{Z}^+} $ on $\mathcal{B}$ such that 
 \begin{align}\label{projection maps}
     P_m\cA(m,n) = \cA(m,n)P_n, \quad \mbox{ for all } m \geq n \geq 0,
 \end{align}
 and
 \begin{align*}
      \cA(m,n) \big|_{\Ker P_n} : \Ker P_n \to \Ker P_m \quad \mbox{ is invertible for } m \geq n \geq 0.
 \end{align*}
 Also, for $ n \geq m \geq 0$, we denote $\cA(m,n) := \big( \cA(n,m)\big|_{\Ker P_m} \big)^{-1} : \Ker P_n \to \Ker P_m$ and $Q_m := Id_{\mathcal{B}} - P_m$ for each $m \in \mathbb{Z}^+$. Now, for $m,n \in \mathbb{Z}^+$, we define another operator 
\begin{equation}\label{def: G}
\mathcal G(m, n):=\begin{cases}
\cA(m, n)P_n & \text{for $m \geq n$;}\\
-\cA(m, n)Q_n & \text{for $m< n$.}
\end{cases}
\end{equation} 
This operator is usually called Green operator. With respect to projection maps $(P_n)_{n \in \mathbb{Z}^+}$, we have $\mathcal{B}= E_n \oplus F_n$ for each $n \in \mathbb{Z}^+$. Here $E_n$ and $F_n$ are ranges of projections $P_n$ and $Q_n$ respectively. Now, let us consider a space $ \cM$ consisting of continuous functions,
\begin{align*}
    \eta: \{ (n, \phi): n \in \mathbb{Z}^+, \phi \in F_n \} \to \mathcal{B}
\end{align*} such that 
\begin{align*}
    \| \eta \|_\infty := \sup \{ \| \eta(n, \phi) \|_{\mathcal{B}}: n \in \mathbb{Z}^+ \mbox{ and } \phi \in F_n  \} < \infty. 
\end{align*}
Clearly, $(\cM,\|\cdot\|_\infty)$ is a Banach space. We also write
\begin{align*}
    \eta^n= \eta(n, \cdot) \quad \mbox{and } \quad h^n= Id_{F_n} + \eta^n.
\end{align*}
Here $Id_{F_n}$ is identity map on subspace $F_n$.
Furthermore, for each $m \in \mathbb{Z}^+$, let $f_m \colon \mathcal{B} \to X$, be a sequence of  maps such that $f_m(0)=0$ for all $ m \in \mathbb{Z}^+$ and there exist numbers $c_m>0$ satisfying
\begin{equation}\label{lipn disc}
|f_m(\phi)-f_m(\psi)| \le c_m \min{(1,\|\phi - \psi\|_{\mathcal{B}})} ,
\end{equation}
for every  $m\in \mathbb{Z}^+$ and $ \phi,\psi \in \mathcal{B} $.
Finally, we consider a \emph{semilinear difference equation} given by
\begin{equation}\label{dis semilin eq}
x(m+1)=A_m x_m + f_m(x_m) \qquad \text{ for all } m\in \mathbb{Z}^+.
\end{equation}
Given $(n, \phi) \in \mathbb{Z}^+ \times \mathcal{B}$, there exists a unique sequence $x:\mathbb{Z} \to X$ such that $x_n= \phi$ and $x$ satisfies the semilinear difference equation (\ref{dis semilin eq}). We write the solution of equation (\ref{dis semilin eq}) in terms of operator $\cR(m,n)$ given by,
\begin{align}\label{general solu.}
    x_m= \cR(m,n) x_n \qquad \mbox{ for all } m \geq n \geq 0.
\end{align}
Now we recall \cite{SM} and \cite{DP} to give the variation of constants formula for a difference equation given by
\begin{align}\label{diff. equ}
    x(m+1) = A_m x_m + p_m, 
\end{align}
where $(p_m)_{m \in \mathbb{Z}^+}$ is a sequence in $X$. Define $\Gamma : \mathbb{Z}^- \to \mathcal{L}(X)$ by
\begin{align*}
    \Gamma(j) = \begin{cases}
Id_X & \text{for $j=0$;}\\
0_X & \text{for $j<0$.}
\end{cases}
\end{align*}
The symbol $\mathcal{L}(X)$ denotes the space of bounded linear operators in X. For $v \in X$, define $\Gamma v: \mathbb{Z}^- \to X$ by
\begin{align*}
    (\Gamma v) (j) := \Gamma(j)v = \begin{cases}
v & \text{for $j=0$;}\\
0 & \text{for $j<0$.}
\end{cases}
\end{align*}
If $x: \mathbb{Z} \to X$ is defined by $x(j)= 0 $ for $j \leq 0$ and $x(j)= v$ for $j >0$, then $x_0 =0$ and $x_1 = \Gamma v$. Since $x_0= 0 \in \mathcal{B}$, by assumption (A), we have that $\Gamma v = x_1 \in \mathcal{B}$ and 
\begin{align}
    \|\Gamma v\|_\mathcal{B} \leq K(1) |v|.
\end{align}
\begin{theorem}\cite{SM}\label{thm:1}
Let $\phi \in \mathcal{B}$. A sequence $x: \mathbb{Z} \to X$ is a solution of $(\ref{diff. equ})$ with initial value $x_0 = \phi$ if and only if for $m \geq 0$, the segment $x_m$ satisfy the following relation in $\mathcal{B}$,
\begin{align}
    x_m = \cA(m,0)\phi + \sum_{k=0}^{m-1} \cA(m,k+1) (\Gamma p_{k}), \quad m \geq 0.
 \end{align}
\end{theorem}

As a corollary of above result, we have that $x:\mathbb{Z} \to X$ is a solution of (\ref{dis semilin eq}) if and only if $\phi = x_0 \in \mathcal{B}$ and 
\begin{align*}
    x_m = \cA(m,0) \phi + \sum_{k=0}^{m-1} \cA(m,k+1) (\Gamma f_{k} (x_{k})), \quad m \geq 0.
\end{align*}
Let $x= x(\cdot; 0, \phi)$ be a solution of equation (\ref{dis semilin eq}) for $\phi \in \mathcal{B}$, then using constants of variation formula from Theorem \ref{thm:1}, for $m \geq n \geq 0$, we have
\begin{align*}
    x_m & =  \cA(m,0) \phi + \sum_{k=0}^{m-1} \cA(m,k+1) (\Gamma f_{k} (x_{k})) \\
    & = \cA(m,0) \phi + \sum_{k=0}^{n-1} \cA(m,k+1) (\Gamma f_{k} (x_{k})) + \sum_{k=n}^{m-1} \cA(m,k+1) (\Gamma f_{k} (x_{k})) \\
    & = \cA(m,n) \big\{  \cA(n,0) \phi + \sum_{k=0}^{n-1} \cA(n,k+1) (\Gamma f_{k} (x_{k})) \big\} \\ & \hspace{6cm}+ \sum_{k=n}^{m-1} \cA(m,k+1) (\Gamma f_{k} (x_{k})) \\
    & = \cA(m,n) x_n + \sum_{k=n}^{m-1} \cA(m,k+1) (\Gamma f_{k} (x_{k})),
\end{align*}
for all $m \geq n \geq 0 $. Therefore, we obtain that $x: \mathbb{Z} \to X$ is a solution of (\ref{dis semilin eq}) if and only if $x_0 \in  \mathcal{B}$ and for $m \geq n \geq 0$ 
\begin{align}\label{variation of const.}
    x_m = \cA(m,n) x_n + \sum_{k=n}^{m-1} \cA(m,k+1) (\Gamma f_{k} (x_{k})).
\end{align}
\section{Main Result} \label{main}
The following theorem is our main result.
\begin{theorem}\label{thm_2}
Assume that the semilinear equation $(\ref{dis semilin eq})$ admits 
\begin{equation}\label{cc disc}
q:=\sup_{m\in \mathbb{Z}^+} \bigg ( \sum_{n\in \mathbb{Z}^+}c_{n} \| \mathcal G(m,n+1) \|  \bigg ), \quad \mbox{ with }\, K(1)q <1 .
\end{equation}
Then there exists a function $ \eta \in \mathcal{M} $, such that for every $m \geq n \geq 0$
 \begin{align}\label{result}
     h^m \circ \cA(m,n) = \cR(m,n) \circ h^n, \quad \mbox{ on } F_n. 
 \end{align}
Moreover, each map $h^n$ is one-to-one provided that the condition $(\ref{cc disc})$ holds with a constant sequence of Lipschitz constants $(c_n)_{n \in \mathbb{Z}^+}$.
\end{theorem}

First we prove a lemma which will be used to establish that $h^n$ is one-one for $n \in \mathbb{Z}^+$. Set
\begin{align}\label{new cc disc 1}
    a_{m,n} := \|\mathcal{G}(m,n)\|.
\end{align}
\begin{lemma}
Assume that the sequence of Lipschitz constants $(c_n)_{n \in \mathbb{Z}^+} $ in~$(\ref{lipn disc})$ is constant and that~$(\ref{cc disc})$  holds.  Then,  for each fixed $m \in \mathbb{Z}^+$, we have that 
\begin{align*}
   \lim_{n \to \infty}a_{m,n}=0.
\end{align*}
\end{lemma}
\begin{proof}
Let $C >0 $ be such that $c_n = C$ for all $n \in \mathbb{Z}^+$.
By~(\ref{cc disc}), it follows that 
\begin{align*}
    \sup_{m\in \mathbb{Z}^+} \bigg ( \sum_{n\in \mathbb{Z}^+}c_{n} \| \mathcal G(m,n+1) \|  \bigg ) = C \sup_{m \in \mathbb{Z}^+}\Big(\sum_{n \in \mathbb{Z}^+} a_{m,n+1}\Big) < \infty.
\end{align*}
Therefore, for each fixed $m \in \mathbb{Z}^+$,
\begin{align*}
    \sum_{n \in \mathbb{Z}^+} a_{m,n+1} < \infty.
\end{align*}
Hence, $\lim_{n \to \infty} a_{m,n+1} =0$ for each fixed $m \in \mathbb{Z}^+$. Equivalently, 
\begin{align}\label{cal 9}
    \lim_{n \to \infty} a_{m,n} =0 \qquad \mbox{ for each fixed } m \in \mathbb{Z}^+.
\end{align}
\end{proof}

Now we give our proof of Theorem \ref{thm_2}.
\begin{proof}
Let  us consider a map $F: \cM \to \cM$ given by
\begin{align}
    F(\eta)(n,\phi) = \sum_{m \in \mathbb{Z}^+} \mathcal{G}(n,m+1)\Gamma f_{m}\big(\cA(m,n)\phi+ \eta^{m}(\cA(m,n)\phi)\big),
\end{align}
where $\eta \in \cM$ and $(n,\phi) \in \mathbb Z^+ \times F_n$. Furthermore, since $f_m(0)= 0$ for each $m \in \mathbb{Z}^+$ and using~(\ref{lipn disc}), we have,
\begin{align*}
|f_m(\phi)| \leq c_m \min{(1, \| \phi \|_{\mathcal{B}})} \leq c_m, \quad \mbox{ for all } \phi \in \mathcal{B}.
\end{align*}
Therefore, using above estimate,
\begin{align*}
    \| F(\eta)\|_\infty \leq \sup_{n \in \mathbb{Z}^+}\Big( K(1)\sum_{m\in \mathbb{Z}^+} c_{m} \| \mathcal{G}(n,m+1) \|\Big)= K(1)q < \infty.
\end{align*}
Therefore, the operator $F$ is well defined. We now claim that $F$ is a contraction map. For each $\eta,\xi \in  \cM$ we have
\begin{align*}
    &|F(\eta)(n,\phi)  - F(\xi)(n,\phi)| \\ & \leq K(1) \sum_{m \in \mathbb{Z}^+} \| \mathcal{G}(n,m+1)\| | f_{m}\big(\cA(m,n)\phi+ \eta^{m}(\cA(m,n)\phi)\big) - \\ & \hspace{5cm} f_{m}\big(\cA(m,n)\phi + \xi^{m}(\cA(m,n)\phi)\big) | \\
    & \leq K(1) \sum_{m \in \mathbb{Z}^+} \| \mathcal{G}(n,m+1)\| c_m | \eta^{m}( \cA(m,n)\phi ) -\xi^{m}(\cA(m,n)\phi)| \\
    & \leq K(1) \sum_{m \in \mathbb{Z}^+} c_m \| \mathcal{G}(n,m+1)\| \| \eta - \xi\|_\infty.
\end{align*}
Therefore,
\begin{align*}
    \| F(\eta)  - F(\xi) \|_{\infty}  \leq K(1) q \,\|\eta -\xi\|_\infty.
\end{align*}
As $K(1)q < 1$, the operator $F$ is a contraction map. Therefore it has a unique fixed point function, say $\eta$ i.e. $F(\eta)= \eta$.
Now, for $\phi \in F_n$, we have
\begin{align}\label{cal 0}
  \eta^n(\phi) & = \sum_{m \in \mathbb{Z}^+} \mathcal{G}(n,m+1) \Gamma f_{m}\big(\cA(m,n)\phi+ \eta^{m}(\cA(m,n)\phi)\big) \nonumber \\
    & = \sum_{m \in \mathbb{Z}^+} \mathcal{G}(n,m+1) \Gamma f_{m}\big( h^{m} (\cA(m,n)\phi) \big).
\end{align}
Note that,
\begin{align}\label{cal 1}
    Q_n\eta^n(\phi) = - \sum_{m \geq n} \cA(n,m+1) Q_{m+1} \Gamma f_{m}\big( h^{m} (\cA(m,n) \phi) \big).
\end{align}
Take $0 \leq p \leq n$. By~(\ref{cal 1}),  for $F_p \ni \psi = \cA(p,n)\phi $ we have that 
\begin{align}\label{cal 2}
    Q_p \eta^p(\psi) = - \sum_{m \geq p} \cA(p,m+1)Q_{m+1} \Gamma f_{m}\big( h^{m}(\cA(m,p)\psi) \big).
\end{align}
Also from (\ref{cal 1}),
\begin{align*}
    \cA(p,n) Q_n\eta^n(\phi) = - \sum_{m \geq n} \cA(p,m+1) Q_{m+1}\Gamma f_{m}\big( h^{m}(\cA(m,n)\phi) \big) .
\end{align*}
Using (\ref{cal 2}) and the above relation, we obtain that
\begin{align*}
    \cA(p,n) Q_n\eta^n(\phi) = Q_p\eta^p & (\cA(p,n) \phi) \\ & + \sum_{m =p}^{n-1} \cA(p,m+1) Q_{m+1} \Gamma f_{m}\big( h^{m} (\cA(m,n)\phi) \big).
\end{align*}
Equivalently,
\begin{align}\label{cal 3}
     Q_n\eta^n(\phi) = \cA(n,p) Q_p\eta^p & (\cA(p,n) \phi) \nonumber \\ & + \sum_{m =p}^{n-1} \cA(n,m+1) Q_{m+1} \Gamma f_{m}\big( h^{m} (\cA(m,n)\phi) \big).
\end{align}
Similarly, from (\ref{cal 0}), it follows that 
\begin{align}\label{cal 4}
    P_n\eta^n(\phi) =  \sum_{m < n} \cA(n,m+1) P_{m+1} \Gamma f_{m}\big( h^{m}(\cA(m,n)\phi) \big).
\end{align}
For $ F_p \ni \psi = \cA(p,n)\phi $, we have that 
\begin{align}\label{cal 5}
    P_p\eta^p(\psi) =  \sum_{m < p} \cA(p,m+1) P_{m+1} \Gamma f_{m}\big( h^{m}(\cA(m,p)\psi) \big).
\end{align}
Hence, 
\begin{align}\label{cal 6}
P_n \eta^n(\phi) = \cA(n,p) P_p & \eta^p (\cA(p,n) \phi) \nonumber \\ & +   \sum_{m = p}^{n-1} \cA(n,m+1) P_{m+1} \Gamma f_{m}\big( h^{m} (\cA(m,n)\phi) \big).    
\end{align}
By adding~(\ref{cal 3}) and (\ref{cal 6}), we obtain that 
\begin{align*}
    \eta^n(\phi) = \cA(n,p) \eta^p (\cA(p,n) \phi) +  \sum_{m = p}^{n-1} \cA(n,m+1) \Gamma f_{m}\big( h^{m} (\cA(m,n)\phi) \big).
\end{align*}
Since $h^n = Id_{F_n} + \eta^n$, we conclude that 
\begin{align*}
    h^n(\phi) = \cA(n,p) \big( h^p( & \cA(p,n)\phi) \big) +  \sum_{m= p }^{n-1} \cA(n,m+1)\Gamma f_{m}\big( h^{m}(\cA(m,n)\phi) \big).
\end{align*}
Moreover, 
\begin{align}\label{cal 7}
    h^n(\cA(n,p)\psi) = \cA(n,p) h^p(\psi) + \sum_{m=p}^{n-1} \cA(n,m+1) \Gamma f_m(h^m(\cA(m,p) \psi))
\end{align}
On the other hand, by the variation of constants formula in (\ref{variation of const.}) and the definition of $\mathcal{R}(m,n)$ in (\ref{general solu.}), for all $n \geq p \geq 0$, we have,
\begin{align*}
    \mathcal{R}(n,p) h^p (\psi) = \cA(n,p) h^p (\psi)  + \sum_{m=p}^{n-1} \cA(n,m+1) \Gamma f_m \big( \mathcal{R}(m,p)h^p (\psi) \big).
\end{align*}
Comparing above relation with~(\ref{cal 7}), and using strong induction on $n$, we obtain,
\begin{align}\label{cal 8}
    \mathcal{R}(n,p) h^p ( \psi) = h^n(\cA(n,p) \psi), \quad \mbox{ for all } n \geq p \geq 0 \mbox{ and } \psi \in F_p.
\end{align}
Hence, we proved that~(\ref{result}) holds.

Now we show that the map $h^p$ is one-to-one for $p \in \mathbb{Z}^+$.
Observe that 
\begin{align*}
    \|\mathcal{G}(p,n)\| = \|\cA(p,n)Q_n\| = a_{p,n} \quad \mbox{ for all } n >p \geq 0.
\end{align*}
Now let $n > p \geq 0$ and $\phi_1, \phi_2 \in F_p$ such that
\begin{align*}
    h^p(\phi_1) = h^p(\phi_2).
\end{align*}
Using above relation and~(\ref{cal 8}), we have,
\begin{align*}
     h^n(\cA(n,p) \phi_1) =  h^n(\cA(n,p) \phi_2).
\end{align*}
Thus,
\begin{align}\label{cal 10}
    \cA(n,p)(\phi_1- \phi_2) = - \big( \eta^n(\cA(n,p) \phi_1) - \eta^n(\cA(n,p) \phi_2) \big).
\end{align}
Note that, as $\phi_1,\phi_2 \in F_p$,
\begin{align*}
\phi_1 - \phi_2 = \cA(p,n) \cA(n,p)Q_p  \big( \phi_1 - \phi_2 \big) = \cA(p,n) Q_n \cA(n,p)(\phi_1-\phi_2).
\end{align*}
Also, $\cA(p,n)Q_n$ is a nonzero map in $\mathcal{B},$ as $\cA(p,n)\big|_{F_n}$ is invertible.
Therefore, we have,
\begin{align*}
    \| \cA(n,p) (\phi_1 - \phi_2) \|_{\mathcal B} & \geq \frac{\| \phi_1 - \phi_2 \|_{\mathcal{B}}}{ \| \cA(p,n) Q_n\| } \\
    & = \frac{\| \phi_1 - \phi_2 \|_{\mathcal{B}}}{a_{p,n}}.
\end{align*}
If $\phi_1 \neq \phi_2$, then it follows from~\eqref{cal 9} that the function $n \to \cA(n,p)(\phi_1 - \phi_2)$ is unbounded. However, the right hand side of~(\ref{cal 10}), is bounded by $2\|\eta\|_{\infty}$, which leads to a contradiction. Therefore $\phi_1= \phi_2$ and hence $h^p$ is one-one for all $p \in \mathbb{Z}^+$.

This completes the proof of our Theorem \ref{thm_2}.
\end{proof}

\subsection{\textbf{Uniform Exponential Dichotomy Case:}}\label{exp_dich}
Let us assume that~(\ref{dis lin eq}) admits uniform exponential dichotomy, i.e. for each $n \in \mathbb{Z}^+$, there exists projection $P_n$  such that:
\begin{enumerate}
    \item $ \cA(m,n)P_n = P_m\cA(m,n)$ for all $m \geq n \geq 0 $;
    \item for $m \geq n$, $\cA(m,n)\big|_{\Ker P_n} : \Ker P_n \to \Ker P_m$ is an invertible map and we denote the inverse map $\big(\cA(m,n)\big|_{\Ker P_n}\big)^{-1}$ by $\cA(n,m)$;
    \item there exists $D, \lambda >0$, such that
    \begin{align*}
        \| \cA(m,n)P_n \| \leq D e^{-\lambda(m-n)} \qquad \mbox{ for all } m \geq n; 
    \end{align*}
    and 
    \begin{align*}
        \| \cA(m,n)\big(Id_\mathcal{B} - P_n \big) \| \leq D e^{-\lambda(n-m)} \qquad \mbox{ for all } m < n.
    \end{align*}
\end{enumerate}
In particular, we have that,
\begin{align*}
    \|\mathcal{G}(m,n) \| \leq D  e^{-\lambda|m-n|} \qquad \mbox{ for all } m,n \in \mathbb{Z}^+.
\end{align*}
By comparing the above expression with~(\ref{new cc disc 1}), we have that  $a_{m,n} \leq D e^{-\lambda |m-n|}$. Finally, let $C>0$ be a constant such that $c_m= C$ for all $m \in \mathbb{Z^+} $. Using above estimates, we have
\begin{align*}
   C \sup_{m \in \mathbb{Z}^+} \Big(\sum_{n \in \mathbb{Z}^+} a_{m,n+1} \Big) & \leq  C \sup_{m \in \mathbb{Z}^+} \Big( \sum_{n \in \mathbb{Z}^+} D e^{-\lambda |m-n-1|} \Big)\\ %& = D \Big( \sum_{n = 0}^{ m-2} e^{-\lambda |m-n-1|} + \sum_{n = m-1}^{+\infty} e^{-\lambda |m-n-1|}  \Big) \\
    & = C\,D \sup_{m \in \mathbb{Z}^+}  \Big( \frac{e^{-\lambda }(1- e^{-\lambda(m-1)})}{1- e^{-\lambda}} + \frac{1}{1-e^{-\lambda}} \Big) < \infty.
\end{align*}
Using the constant $C$, we can make sure that equation (\ref{cc disc}) is satisfied. Therefore, if the linear delay difference equation (\ref{dis lin eq}) satisfies uniform exponential dichotomy then our Theorem \ref{thm_2} is applicable.

\paragraph{\textbf{Acknowledgements}} Author would like to thanks Prof. D. Dragičević for his valuable suggestions throughout the process of solving and writing this article. The Author is supported by Croatian Science Foundation under the project IP-2019-04-1239.

%
% ---- Bibliography ----
%

\end{document}